\documentclass[11pt,reqno]{amsart}
\usepackage{graphicx}
\usepackage{amssymb,amsmath}
\usepackage{amsthm}
\usepackage{setspace}
\usepackage{color,graphicx}
\usepackage{hyperref}
\usepackage{color}
\usepackage{epstopdf}

\newcommand{\be}{\begin{equation}}

\newcommand{\ee}{\end{equation}}

\newcommand{\cn}{{\rm \,cn}}
\newcommand{\sn}{{\rm \,sn}}
\newcommand{\dn}{{\rm \,dn}}
\newcommand{\sech}{{\rm \,sech}}

\numberwithin{equation}{section}
\numberwithin{figure}{section}

\newtheorem{theorem}{Theorem}[section]
\newtheorem{proposition}[theorem]{Proposition}

\newtheorem{lemma}[theorem]{Lemma}
\newtheorem{corollary}[theorem]{Corollary}
\newtheorem{definition}[theorem]{Definition}

\begin{document}

\vglue-1cm \hskip1cm
\title[Spectral Stability for Double Dispersion Equation]{Spectral Stability of Periodic Traveling Wave Solutions for a Double Dispersion Equation}

\begin{center}



\maketitle

{\bf F\'abio Natali}

{Departamento de Matem\'atica - Universidade Estadual de Maring\'a\\
Avenida Colombo, 5790, CEP 87020-900, Maring\'a, PR, Brazil.}\\
{ fmanatali@uem.br}

\vspace{3mm}

{\bf Thiago Pinguello de Andrade }

{Departamento de Matem\'atica - Universidade Tecnol\'ogica Federal do Paran\'a\\
Av. Alberto Corazzai, 1640, CEP 86300-000, Corn\'elio Proc\'opio, PR, Brazil.}\\
{thiagoandrade@utfpr.edu.br}

\end{center}

\begin{abstract}
In this paper, we investigate the spectral stability of periodic traveling waves for a cubic-quintic and double dispersion equation. Using the quadrature method we find explict periodic waves and we also present a characterization for all positive and periodic solutions for the model using the monotonicity of the period map in terms of the energy levels. The monotonicity of the period map is also useful to obtain the quantity and multiplicity of non-positive eigenvalues for the associated linearized operator and to do so, we use tools of the Floquet theory. Finally, we prove the spectral stability by analysing the difference between the number of negative eigenvalues of a convenient linear operator restricted to the space constituted by zero-mean periodic functions and the number of negative eigenvalues of the matrix formed by the tangent space associated to the low order conserved quantities of the evolution model.   \end{abstract}

\section{Introduction.}

Our aim is to prove existence and spectral stability of periodic traveling wave solution for a cubic-quintic and double dispersion equation (DDE henceforth) given by
\begin{equation}
\label{ddequation}
u_{tt}-u_{xx}+h_1u_{xxxx}-h_2u_{ttxx}+(f(u))_{xx}=0,
\end{equation}
where $u=u(x,t)$, with $(x,t)\in \mathbb{R}\times \mathbb{R}^+$, $h_1$ and $h_2$ are positive real constants and
$$f(u)=au^{p+1}+bu^{2p+1},$$
with $p>0$ and $a,b \in \mathbb{R}\backslash\{0\}$. We are interested in the special case $p=2$ and $a=b=1$. We also assume that $h_1=5$ and $h_2=4$ for technical reasons. All of those specific cases enable us to consider equation \eqref{ddequation} as
\begin{equation}
\label{ddeparticular}
u_{tt}-u_{xx}+5u_{xxxx}-4u_{ttxx}+(u^3+u^5)_{xx}=0.
\end{equation}
\indent The evolution equation $(\ref{ddeparticular})$ does not enjoy of a scaling invariance in order to determine a sharp Sobolev space to obtain the existence of a global well-posedness result in time and the reason for that is because of the presence of the double dispersion term $5u_{xxxx}-4u_{xxtt}$. Moreover, it can be reduced in the following nonlinear system as
\begin{equation}\label{sysdde}
	\left\{\begin{array}{llll}u_t=v_x,\\
		v_t=(I-4\partial_x^2)^{-1}\partial_x((I-5\partial_x^2)u-(u^3+u^5)),
	\end{array}\right.\end{equation}
which gives us the natural conserved quantities (energy and momentum, respectively) given by 
\begin{equation}\label{E}
	E(u,v)=\frac{1}{2}\int_0^L\left(u^2+5u_x^2+v^2+4v_x^2-\frac{1}{2}u^4-\frac{1}{3}u^6\right)dx,
\end{equation}
and
\begin{equation}\label{F}
	F(u,v)=\int_0^L(uv+4u_xv_x)dx,
\end{equation}
where $(u,v):=\left(\begin{array}{lll}u\\ v\end{array}\right)$. From system $(\ref{sysdde})$, we also have the following conserved quantities 
\begin{equation}\label{G}
	G(u,v)=\int_0^Ludx,
\end{equation}
and 
\begin{equation}\label{H}
	H(u,v)=\int_0^Lvdx.
\end{equation}

In addition, the nonlinear system $(\ref{sysdde})$ can be reduced in an abstract Hamiltonian system as
\begin{equation}\label{ham}
			\left(\begin{array}{ll}u\\ v\end{array}\right)_t=J\left(\begin{array}{ll} \partial_u E \\ \partial_v E \end{array}\right),
\end{equation}
where $\partial_u E$ and $\partial_v E$ denote the Fr\'echet derivative of $E$ with respect to $u$ and $v$, respectively. Operator $J$ is defined as $J=(I-4\partial_x^2)^{-1}\left(\begin{array}{lll}0\ \ \partial_{x}\\
\partial_{x}\ \ 0\end{array}\right)$.\\
\indent Because of the natural translation invariance present in the equation $(\ref{sysdde})$, we can determine traveling wave solutions of the form $(u(x,t),v(x,t))=(\phi(x-ct),\psi(x-ct))$, where $\phi, \psi:\mathbb{R}\rightarrow \mathbb{R}$ are real-valued functions which are $L-$periodic at the space variable and $c$ is called the wave-speed of the periodic solution which is assumed to be positive. Replacing this kind of solution into equation \eqref{sysdde}, we see that the pair $(\phi,\psi)$ must satisfy the following system of ODE,
\begin{equation}\label{sysdde-ode}
	\left\{\begin{array}{lllll}-c\phi'=\psi'\\
		-c(\psi'-4\psi''')=\phi'-5\phi'''-(\phi^3+\phi^5)',
	\end{array}\right.\end{equation}
\indent System $(\ref{sysdde-ode})$ can be reduced into a single equation as
$$c^2\phi'-\phi'+5\phi'''-4c^2\phi '''+(\phi^3+\phi^5)'=0.$$
Integrating the last ODE, we obtain
$$(c^2-1)\phi+(5-4c^2)\phi''+\phi^3+\phi^5-A=0,$$
where $A$ is a constant of integration. In order to simplify the notation, we consider $r=1-c^2$ and $s=5-4c^2$. Assuming $s\neq0$, one has 
\begin{equation}
\label{ddedo}
-\phi''+\frac{r}{s}\phi -\frac{1}{s}(\phi^3+\phi^5)+\frac{A}{s}=0.
\end{equation}
Equation $(\ref{ddedo})$ can be seen as a two-parameter function depending on the pair $(c,A)$. In order to obtain explicit periodic waves $(\phi,\psi)$ that solve $(\ref{sysdde-ode})$, we are going to assume $c\in (0,1)$ and $A=0$. Thus, we have the final equation 
\begin{equation}
\label{ddedo1}
-\phi''+\frac{r}{s}\phi -\frac{1}{s}(\phi^3+\phi^5)=0.
\end{equation}

\indent The periodic traveling solution $\overrightarrow{\phi_c}=(\phi,-c\phi)$ which solves equation $(\ref{sysdde-ode})$ arises as a critical point of the augmented energy functional $D_{c}$ defined as $D_c(u,v)=E(u,v)+cF(u,v)$. It is well known that the spectral stability of $\overrightarrow{\phi_c}=(\phi,-c\phi)$ can be determined by minimizing the functional $E$ under fixed constraints $F$, $G$ and $H$. Therefore, since $\overrightarrow{\phi_c}$ is a critical point of $D_c$, it is suitable to think that the spectral stability can be determined by proving that the second derivative of $D_c$ at the point $\overrightarrow{\phi_c}$, and denoted by 
\begin{equation}\label{operator}\mathcal{D}:=D_c''(\overrightarrow{\phi_c})=\left(\begin{array}{cc} -5\frac{\partial^2}{\partial x^2}+1-3\phi^2-5\phi^4 \,\ \,\ & -4c\frac{\partial^2}{\partial x^2} +c\\\\
	-4c\frac{\partial^2}{\partial x^2} +c&  -4\frac{\partial^2}{\partial x^2} +1\end{array}\right)\end{equation}
is strictly positive over the set $\{(-c4\phi''+c\phi,4\phi''-\phi),(1,0),(0,1)\}^{\bot}$.\\
\indent Without further ado, we shall present the basic setting to prove our spectral stability result. First, let us consider the existence of a time $T>0$ such that the pair of solutions $(u,v)$ of $(\ref{sysdde})$ is an element of $C([0,T),\mathbb{H}_{\text{per}}^1)$ (see Lemma $\ref{cauchysysdde}$). Next, let us consider the change of variables:
\begin{equation}\label{changevar1}
	p(x,t)=u(x+ct,t)-\phi(x)\ \ \ \ \ \ \ \ \mbox{and}\ \ \ \ \ \ \ q(x,t)=v(x+ct,t)+c\phi(x),
\end{equation}
We obtain by substituting $(\ref{changevar1})$ into $(\ref{sysdde})$ and after neglecting all high order non-linear terms in $p$ and $q$, the following linear evolution system
\begin{equation}\label{sys1}
	\left\{\begin{array}{lllll}
		p_t=\partial_x(q+cp),\\
		q_t=(1-4\partial_x^2)^{-1}\partial_x((1-5\partial_x^2)p-(3\phi^2+5\phi^4)p+c(1-4\partial_x^2)q).
	\end{array}\right.
\end{equation}
Since $\partial_x(q+cp)=(1-4\partial_x^2)^{-1}\partial_x((1-4\partial_x^2)q+c(1-4\partial_x^2)p)$, we obtain by $(\ref{sys1})$, the following linear system
\begin{equation}\label{sys2}
	W_t=J\mathcal{D}W,
\end{equation}
where $W=\left(\begin{array}{cccc} p\\ q\end{array}\right)=(p,q)$, $J=(I-4\partial_x^2)^{-1}\left(\begin{array}{lll}0\ \ \partial_{x}\\
	\partial_{x}\ \ 0\end{array}\right)$ and $\mathcal{D}$ is defined as in $(\ref{operator})$. To obtain the standard spectral problem, we need to consider the separation of variables $W(x,t)=e^{\mu t}V(x)$ to generate a growing mode solution that determines if the deviations in $(\ref{changevar1})$ are stable or not in the following sense: if $\mu \in \mathbb{C}$ is purely imaginary, then the norm of $W$ in $\mathbb{H}_{per}^1$ is bounded and the deviations in $(\ref{changevar1})$ are said to be stable. Otherwise, if the real part of $\mu$ is positive, the norm of $W$ in $\mathbb{H}_{per}^1$ is large enough when $t$ increases and the deviations in $(\ref{changevar1})$ can be considered unstable. The case when $\mu$ has negative real part is not considered in our model since equation $(\ref{ddequation})$ does not have dissipative terms and $(\ref{E})-(\ref{H})$ are conserved quantities. From $(\ref{sys2})$, the corresponding spectral problem associated to the equation $(\ref{sys2})$ then becomes
\begin{equation}\label{sys3}
	J\mathcal{D}V=\mu V,
\end{equation}
\indent In the periodic context, the spectral problem $(\ref{sys3})$ can be seen in an equivalent form as
\begin{equation}\label{sys4}
	J\mathcal{D}_{\Pi}V=\mu V,
\end{equation}
where 

\begin{equation}\label{operatorproj}\mathcal{D}_{\Pi}=\left(\begin{array}{cc} -5\frac{\partial^2}{\partial x^2}+1-3\phi^2-5\phi^4+\frac{1}{L}(3\phi^2+5\phi^4,\cdot)_{L_{per}^2} \,\ \,\ & -4c\frac{\partial^2}{\partial x^2} +c\\\\
		-4c\frac{\partial^2}{\partial x^2} +c&  -4\frac{\partial^2}{\partial x^2} +1\end{array}\right),\end{equation}
is the projection operator defined over the space $$\mathbb{M}=\left\{(f,g)\in \mathbb{L}_{per}^2;\ \int_0^Lfdx=0\ \mbox{and}\ \int_0^Lgdx=0 \right\}.$$
We have the following definition of spectral stability associated to the problem $(\ref{sys4})$.
\begin{definition}\label{defistab1}
	The periodic wave $\overrightarrow{\phi_c}=(\phi,-c\phi)\in \mathbb{H}_{per}^2$ is said to be spectrally stable if $\sigma(J\mathcal{D}) \subset i\mathbb{R}$ in  $\mathbb{M}$. Otherwise, that is, if $\sigma(J\mathcal{D})$ in $\mathbb{M}$ contains a point $\mu$ with $Re(\mu)>0$, the periodic wave $\overrightarrow{\phi_c}$ is said to be spectrally unstable.
\end{definition}

\indent It is important to mention that if one considers the spectral problem $(\ref{sys4})$ in the space $\mathbb{M}$, the operator $J$ becomes invertible with a bounded inverse, and the problem $(\ref{sys4})$ can be rewritten as
\begin{equation}\label{sys5}
	\mathcal{D}_{\Pi}V=\mu J^{-1}V,
\end{equation}
where $V$ is a smooth periodic vectorial function in $\mathbb{M}$. If the kernel of $\mathcal{D}$ is simple,  the classical theories in \cite{bronski}, \cite{kapitulodeco}, \cite{grillakis1}, \cite{grillakis2}, \cite{hakkaev1},  \cite{kapitula}, \cite{stan} can be applied to obtain the spectral stability of the wave $\overrightarrow{\phi_c}$ according to the Definition $\ref{defistab1}$. In fact, consider the matrix $\mathcal{P}$ defined as

$$
	\mathcal{P}=\left(\begin{array}{llll}\langle \mathcal{D}^{-1}(-c\chi,\chi),(-c\chi,\chi)\rangle_{\mathbb{L}_{per}^2} & \langle \mathcal{D}^{-1}(1,0),(-c\chi,\chi)\rangle_{\mathbb{L}_{per}^2} & \langle \mathcal{D}^{-1}(0,1),(-c\chi,\chi)\rangle_{\mathbb{L}_{per}^2}\\\\
	\langle \mathcal{D}^{-1}(-c\chi,\chi),(1,0)\rangle_{\mathbb{L}_{per}^2} & \langle \mathcal{D}^{-1}(1,0),(1,0)\rangle_{\mathbb{L}_{per}^2} & \langle \mathcal{D}^{-1}(0,1),(1,0)\rangle_{\mathbb{L}_{per}^2}\\\\
	\langle \mathcal{D}^{-1}(-c\chi,\chi),(0,1)\rangle_{\mathbb{L}_{per}^2} & \langle \mathcal{D}^{-1}(1,0),(0,1)\rangle_{\mathbb{L}_{per}^2} & \langle \mathcal{D}^{-1}(0,1),(0,1)\rangle_{\mathbb{L}_{per}^2}\end{array}\right),
$$
where $\chi=4\phi''-\phi$. We have to notice that since $\mathcal{D}$ is self-adjoint, $\ker(\mathcal{D})=[(\phi',-c\phi')]$ and $\{(-c\chi,\chi),(1,0),(0,1)\}\subset \ker(\mathcal{D})^{\bot}=\mbox{range}(\mathcal{D})$, we obtain that the inverse $\mathcal{D}^{-1}$ operator of $\mathcal{D}$ present in the entries of the matrix $\mathcal{P}$ above is well defined because $\mathcal{D}:\ker(\mathcal{D})^{\bot}\rightarrow \ker(\mathcal{D})^{\bot}$. Therefore, the spectral stability of the periodic wave $\overrightarrow{\phi_c}$ is determined if the difference
\begin{equation}\label{difspect}
	n(\mathcal{D}_{\Pi})-n(\mathcal{P})
\end{equation}
is zero. Here, $n(\mathcal{A})$ indicates the number of negative eigenvalues (counting multiplicities) of a certain linear operator $\mathcal{A}$. On the other hand, if the difference in $(\ref{difspect})$ is an odd number, the periodic wave is spectrally unstable. The difference in $(\ref{difspect})$ is commonly called in the current literature as \textit{Hamiltonian Krein Index} and it appears in several spectral stability approaches as in \cite{bronski}, \cite{kapitulodeco}, \cite{hakkaev1}, \cite{kapitula} and \cite{stan}.\\
\indent Particularly in \cite{bronski}, \cite{hakkaev1} and \cite{stan} the authors proved spectral stability results for a general second order PDE using the quadratic pencils technique in order to obtain a precise counting for the Hamiltonian Krein Index $\mathcal{K}_{\rm Ham}$. They obtained that if $\mathcal{K}_{\rm Ham}=0$, the periodic wave is spectrally stable while if $\mathcal{K}_{\rm Ham}$ is an odd number, the periodic wave is spectrally unstable and the senses of spectral stability/instability for both models are very close to the one presented in Definition $\ref{defistab1}$. Among the applications presented by the authors using this technique are the Klein-Gordon and Boussinesq type equations.\\
\indent As far as we can see and in the periodic context, we can apply, since $J=(I-4\partial_x^2)^{-1}\left(\begin{array}{lll}0\ \ \partial_{x}\\
	\partial_{x}\ \ 0\end{array}\right)$ has derivatives in its entries, the classical approaches of stability and studying the equivalent spectral problem $(\ref{sys4})$ to obtain the spectral stability without using the quadratic pencils technique as determined, for example, in \cite{hakkaev1} for the Boussinesq type equations. We have to notice that the equivalent problem $(\ref{sys4})$ can be rewritten, since $J$ in the zero-mean periodic space has bounded inverse, in the form $(\ref{sys5})$ and the spectral stability can be determined by studying the difference $(\ref{difspect})$ as determined in \cite{grillakis2} and \cite{kapitula}. \\
	\indent In the case of waves that vanish at infinity, the authors in \cite{nikolai} have determined the orbital stability for the general equation $(\ref{ddequation})$. More specifically, they considered the nonlinearity $f(u)=au^{p+1}+bu^{2p+1}$ in some cases where it is possible to obtain a hyperbolic secant profile. When $f(u)=u^{3}+u^{5}$, $h_1=5$ and $h_2=4$ in equation $(\ref{ddequation})$, the hyperbolic secant profile becomes the solution given by,
\begin{equation}
\phi(x)=2(1-c^2)^\frac{1}{2}\left( 1+\sqrt{1-\frac{16}{3}(1-c^2)}\cosh\left(2\sqrt{\frac{1-c^2}{5-4c^2}}x\right)\right)^{-\frac{1}{2}}.
\label{solitwav}
\end{equation}
\indent To prove the orbital stability, the authors used the approach in \cite{grillakis1}. They proved that the linearized operator around the solitary wave $\mathcal{D}$ has only one negative eigenvalue, which is simple, and zero is a simple eigenvalue whose associated eigenfunction is $(\phi',-c\phi')$. Both of these facts can then be used to conclude the orbital stability, provided that $d''(c)>0$, where $d''(c)$ is the second derivative with respect to $c$ of the function $d(c)=F(\phi,-c\phi)$. To obtain the positiveness of $d''(c)$, they used a numerical approach. Unfortunately, they obtained some cases where $d''(c)<0$ using numerics, but they were not enabled to conclude an instability result since $J=(I-4\partial_x^2)^{-1}\left(\begin{array}{lll}0\ \ \partial_{x}\\
\partial_{x}\ \ 0\end{array}\right)$ does not have a bounded inverse (the instability theory in \cite{grillakis1} can not be applied). Our intention is to provide a positive answer for the (spectral) stability/instability in the periodic context.\\
\indent Now, we present our spectral stability result:
\begin{theorem}
Let $L > 2\pi\sqrt{\frac{\sqrt{5}}{\sqrt{5}-1}}$ be fixed and consider  
	$c\in \left(0,\frac{1}{2}\sqrt{5-\frac{L^4}{(L^2-4\pi^2)^2}}\right)$. There exists an $c(L)>0$ such that:\\
	(i) if $c\in(0,c(L)]$, the periodic wave  
	$\overrightarrow{\phi_c}=(\phi,-c\phi)$ is spectrally unstable.\\
	(ii) If $c\in \left(c(L),\frac{1}{2}\sqrt{5-\frac{L^4}{(L^2-4\pi^2)^2}}\right)$, the periodic wave  
	$\overrightarrow{\phi_c}=(\phi,-c\phi)$ is spectrally stable.
	
\label{mainT}\end{theorem}
\indent Our paper is organized as follows: in Section 2, we show the existence of a family of periodic wave solutions for the equations $(\ref{edodphi})$. The characterization of periodic waves is determined in Section 3. Spectral analysis for the linearized operator $\mathcal{D}$ is established in Section 4. Finally, the spectral stability of the periodic wave will be shown in Section 5.\\

\indent \textbf{Notation.} Here we introduce the basic notation concerning the periodic Sobolev spaces and other useful notations used in our paper. For a more complete introduction to these spaces, we refer the reader to see \cite{Iorio}. By $L^2_{per}:=L^2_{per}([0,L])$, $L>0$, we denote the space of all square integrable functions which are $L$-periodic. For $s\geq0$, the Sobolev space
$H^s_{per}:=H^s_{per}([0,L])$
is the set of all periodic distributions such that
$$
\|f\|^2_{H^s_{per}}:= L \sum_{k=-\infty}^{\infty}(1+|k|^2)^s|\hat{f}(k)|^2 <\infty,
$$
where $\hat{f}$ is the periodic Fourier transform of $f$. The space $H^s_{per}$ is a  Hilbert space with natural inner product denoted by $(\cdot, \cdot)_{H^s_{per}}$. When $s=0$, the space $H^s_{per}$ is isometrically isomorphic to the space  $L^2_{per}$, that is, $L^2_{per}=H^0_{per}$ (see, e.g., \cite{Iorio}). The norm and inner product in $L^2_{per}$ will be denoted by $\|\cdot \|_{L^2_{per}}$ and $(\cdot, \cdot)_{L^2_{per}}$. In our paper, we do not distinguish if the elements in $H_{per}^s$ are complex- or real-valued.\\ 
\indent In addition, to simplify notation we set
 $\mathbb{H}^s_{per}:= H^s_{per} \times H^s_{per}$
endowed with their usual norms and scalar products. When necessary and since $\mathbb{C}$ can be identified with $\mathbb{R}^2$, notations above can also be used in the complex/vectorial case in the following sense: for $\overrightarrow{f}\in \mathbb{H}_{per}^s$ we have $f=f_1+if_2\equiv (f_1,f_2)$, where $f_i\in H_{per}^s$, $i=1,2$.\\
\indent Throughout this paper, we  fix the following embedding chains with the Hilbert space $L_{per}^2$ identified with its dual (by the Riesz Theorem) as $$H_{per}^1\hookrightarrow L_{per}^2\equiv (L_{per}^2)'\hookrightarrow H_{per}^{-1}.$$
\indent The symbols $\sn(\cdot, k), \dn(\cdot, k)$ and $\cn(\cdot, k)$ represent the Jacobi elliptic functions of \textit{snoidal}, \textit{dnoidal}, and \textit{cnoidal} type, respectively. For $k \in (0, 1)$, $K(k)$ and $E(k)$ will denote the complete elliptic integrals of the first and second kind, respectively. For the precise definition and additional properties of the elliptic functions we refer the reader to  \cite{byrd}.

\section{Explicit Solutions via Quadrature Method.}

Our intention is to apply the quadrature method and obtaining explicit positive periodic solutions for the system \eqref{sysdde-ode}. Indeed, multiplying \eqref{ddedo1} by $-2\phi'$, we see
$$(\phi'^2)'-\frac{r}{s}(\phi^2)'+ \frac{1}{2s}(\phi^4)'+\frac{1}{3s}(\phi^6)'=0.$$
After an integration, it follows that
$$(\phi')^2-\frac{r}{s}\phi^2+ \frac{1}{2s}\phi^4+\frac{1}{3s}\phi^6-B=0,$$
where $B$ is a constant of integration. Thus,
\begin{equation}
\label{eqquadrature}
(\phi')^2=\frac{r}{s}\phi^2- \frac{1}{2s}\phi^4-\frac{1}{3s}\phi^6+B.
\end{equation}
\indent Let us consider $\Psi=\phi^2$. By $(\ref{eqquadrature})$, we obtain
$$(\Psi')^2=4\phi^2(\phi')^2=4\Psi\left(\frac{r}{s}\Psi- \frac{1}{2s}\Psi^2-\frac{1}{3s}\Psi^3+B \right),$$
that is,
$$(\Psi')^2=\frac{4}{3s}\left(-\Psi^4 - \frac{3}{2}\Psi^3 +  3r\Psi^2+3sB\Psi \right).$$
Denoting $R(\Psi)=-\Psi^4- \frac{3}{2}\Psi^3 +  3r\Psi^2 + 3sB\Psi$, we can write
\begin{equation}
\label{eqquadrature2} (\Psi')^2=\frac{4}{3s}R(\Psi).
\end{equation}
Equation $(\ref{eqquadrature2})$ is useful to deduce explicit periodic solutions depending on the  Jacobi elliptic functions. To do so, we need to perform a precise study concerning the real roots of the polynomial $R(\Psi)$.

In what follows, we consider four roots of $R(\Psi)$ satisfying $\alpha_1<\alpha_2=0<\alpha_3<\Psi<\alpha_4$. Important to notice that all those real roots must satisfy the system
\begin{equation}
\label{systemalpha}
\left\{\begin{array}{l}
\alpha_1+\alpha_3+\alpha_4=-\frac{3}{2}\\
\alpha_1\alpha_3+\alpha_1\alpha_4+\alpha_3\alpha_4=-3r\\
\alpha_1\alpha_3\alpha_4=3sB.
\end{array} \right.
\end{equation}

Having in mind the chain of inequalities $\alpha_1<0<\alpha_3<\Psi<\alpha_4$ and the system $(\ref{systemalpha})$, we can apply Formula 257.00 in \cite{byrd} to get
$$\Psi(x)=\frac{\alpha_4dn^2\left(\frac{2}{g\sqrt{3s}}x,k\right)}{1+\beta^2sn^2\left(\frac{2}{g\sqrt{3s}}x,k\right) },$$
where 
\begin{equation}
\label{betagk}
\beta^2=\frac{\alpha_4}{-\alpha_1}k^2>0, \,\   \,\ g=\frac{2}{\sqrt{\alpha_4(\alpha_3-\alpha_2)}}, \,\ \,\  k^2=\frac{-\alpha_1(\alpha_4-\alpha_3)}{\alpha_4(\alpha_3-\alpha_1)},
\end{equation}
where $sn$ and $dn$ are the Jacobi Elliptic functions of senoidal and dnoidal  type, respectively. Parameter $k\in(0,1)$ is called elliptic modulus.  

Thus, $\phi$ is expressed as
\begin{equation}\label{phix}
\phi(x)=\frac{\sqrt{\alpha_4}dn\left(\frac{2}{g\sqrt{3s}}x,k\right)}{\sqrt{1+\beta^2sn^2\left(\frac{2}{g\sqrt{3s}}x,k\right) }}.
\end{equation}
Since the dnoidal function has fundamental period equals to $2K(k)$, one has that the period $T:=T_\phi$ of $\phi$ is given by
\begin{equation}
\label{Tphialphas}
T=\frac{2\sqrt{3s}K(k)}{\sqrt{\alpha_4(\alpha_3-\alpha_1)}},
\end{equation}
where $K(k)=\int_0^1\frac{dt}{\sqrt{(1-k^2t^2)(1-t^2)}}$ is the complete elliptic integral of the first type.

Let $c\in(0,1)$ be fixed. By \eqref{systemalpha}, we can write $\alpha_1$ and $\alpha_2$ in terms of $\alpha_4$ as 
\begin{equation}\label{alpha3alpha4}
\alpha_1=\frac{-2\alpha_4-3-\sqrt{3}\sqrt{\beta}}{4} \,\ \text{and}\,\ \alpha_3=\frac{-2\alpha_4-3+\sqrt{3}\sqrt{\beta}}{4},
\end{equation}
where $\beta=-4\alpha_4^2-4\alpha_4+3+16r$. To make $\alpha_1$ and $\alpha_3$ well defined, it is necessary to consider $\beta\geq 0$. This fact occurs provided that
$$\frac{-1-\sqrt{4+16r}}{2}\leq \alpha_4\leq \frac{-1+\sqrt{4+16r}}{2}.$$

Moreover, in the interval $\left(0,\frac{-1+\sqrt{4+16r}}{2}\right)$, we have $\alpha_3'(\alpha_4)<0$. In addition, condition $\alpha_1<0<\alpha_3<\alpha_4$ implies that the maximum value for the root $\alpha_3$ is $\frac{-1+\sqrt{1+4r}}{2}$,  when $\alpha_4=\frac{-1+\sqrt{1+4r}}{2}$ and the minimum value is $0$, when $\alpha_4=\frac{3+\sqrt{9+48r}}{4}$. We can conclude that $\alpha_3$ and $\alpha_4$ must satisfy the following inequalities
$$0<\alpha_3<\frac{-1+\sqrt{1+4r}}{2}<\alpha_4<\frac{-3+\sqrt{9+48r}}{4}.$$
Therefore, since inequality 
$$\frac{-3+\sqrt{9+48r}}{4}<\frac{ -2+\sqrt{16+64r}}{4}=\frac{ -1+\sqrt{4+16r}}{2}$$
holds, we can deduce $\beta>0$ as desired.

On the other hand, we expressed $\alpha_1$ and $\alpha_3$ in terms of of $\alpha_4$ and all of them are well defined in the interval $\left(\frac{-1+\sqrt{1+4r}}{2},\frac{-3+\sqrt{9+48r}}{4}\right)$. This fact implies that it is possible to express constant $B$ in $(\ref{systemalpha})$,  the period $T$ and the modulus $k$  as functions of the root $\alpha_4$ and $c$, that is,
\begin{equation}
\label{Bfunctionalpha4}
B=\frac{1}{3s}\left(\alpha_4^3+\frac{3}{2}\alpha_4^2-3r\alpha_4\right), \,\ \,\  T=\frac{2\sqrt{2\sqrt{3}s}K(k)}{\sqrt{\alpha_4\sqrt{\beta}}},
\end{equation}
and
\begin{equation}\label{kfunctionalpha4}
k^2=\frac{6\alpha_4^2+9\alpha_4-12r+\alpha_4\sqrt{3}\sqrt{\beta}}{2\sqrt{3}\alpha_4\sqrt{\beta}}.
\end{equation}

Let $c\in(0,1)$ be fixed. In the next theorem, we prove that the period $T$ is strictly increasing as function of $\alpha_4\in\left(\frac{-1+\sqrt{1+4r}}{2},\frac{-3+\sqrt{9+48r}}{4}\right)$. As a consequence, if $\alpha_4 \to \frac{-1+\sqrt{1+4r}}{2}$, then 
$$T \to \frac{2\pi \sqrt{s}}{\sqrt{1+4r-\sqrt{1+4r}}} \,\ \,\ \text{and}\,\ \,\ T>  \frac{2\pi \sqrt{s}}{(1+4r)^\frac{1}{4}\sqrt{\sqrt{1+4r}-1}}.$$
Indeed $\alpha_4 \to \frac{-1+\sqrt{1+4r}}{2}$ implies $\alpha_3 \to \frac{-1+\sqrt{1+4r}}{2}$ and $\alpha_1\to -\frac{1+2\sqrt{1+4r}}{2}$. By the expression of $k$ in \eqref{betagk}, we have $k\to 0$, so that $K(k) \to \frac{\pi}{2}$. Thus, by \eqref{Tphialphas} one has $T \to \frac{2\pi \sqrt{s}}{\sqrt{1+4r-\sqrt{1+4r}}}$. Since ${\rm dn}(\cdot, 0^+) \sim 1$, we obtain that the periodic solution $\phi$ converges to the equilibrium solution of \eqref{ddedo} as
$\phi(x)=\sqrt{\frac{\sqrt{1+4(1-c^2)}-1}{2}}.$
On the other hand, if $\alpha_4 \to \frac{-3+\sqrt{9+48r}}{4}=:\sigma_4$ then
$\alpha_1 \to \frac{-3-\sqrt{9+48r}}{4}=:\sigma_1$, and $\alpha_3 \to 0$.  By the expression of $k$ in \eqref{betagk}, one has $k\to 1$, $K(k) \to \infty$ and $T\rightarrow+\infty$. Since ${\rm dn}(\cdot, 1^-)\sim {\rm sech}(\cdot)$ and ${\rm sn}(\cdot, 1^-)\sim {\rm tanh}(\cdot)$, the periodic wave $\phi$ must be the solitary wave
$$\phi(x)=\left(\frac{-\sigma_1\sigma_4\sech^2(\gamma x)}{-\sigma_1+\sigma_4\tanh^2(\gamma x)}\right)^\frac{1}{2}=\left(\frac{-\sigma_1\sigma_4}{-\sigma_4+\frac{\sigma_4-\sigma_1}{2}+\frac{\sigma_4-\sigma_1}{2}\cosh(2\gamma x)}\right)^\frac{1}{2}, $$
where 
$$\gamma=\frac{2}{g\sqrt{3s}}=\sqrt{\frac{-\sigma_1\sigma_4}{3s}}.$$
\indent Since, $-\sigma_1\sigma_4=3r$,  $\gamma=\sqrt{\frac{r}{s}}$ and $\frac{\sigma_4-\sigma_1}{2}=\frac{\sqrt{9+48r}}{4}$, $\phi$ can be rewritten as
$$\phi(x)=\left(\frac{12r}{3+\sqrt{9+48r}\cosh(2\gamma x)}\right)^\frac{1}{2}.$$
Finally, it is possible to express $\phi$ as
$$\phi(x)=2(1-c^2)^\frac{1}{2}\left( 1+\sqrt{1-\frac{16}{3}(1-c^2)}\cosh\left(2\sqrt{\frac{1-c^2}{5-4c^2}}x\right)\right)^{-\frac{1}{2}}.$$
This last expression, is exactly the solitary wave obtained in \cite{nikolai} for the case $h_1=5$, $h_2=4$ and $a=b=1$. Next theorem establishes the monotonicity of the period $T$ in $(\ref{Bfunctionalpha4})$ in terms of $\alpha_4$ and $c$.

\begin{lemma}\label{teo01}
	Consider $T$ in \eqref{Bfunctionalpha4} as the period-map of the function $\phi$ in \eqref{phix} which depends on the variables $(\alpha_4,c)\in \Omega$, where $\Omega\subset\mathbb{R}^2$ is the open subset
	$$\Omega =  \left(\frac{-1+\sqrt{1+4r}}{2}, \frac{-3+\sqrt{9+48r}}{4}\right) \times \left(0,1\right).$$
	 For all $(\alpha_4,c)\in \Omega$, the period-map $T$ is strictly increasing with respect to $\alpha_4$ and $c$.
\end{lemma}
\begin{proof}
	Computing the derivative of $T$ with respect to $\alpha_4:=\alpha $, we obtain
	\begin{equation}\label{derivTalpha4}
	\frac{\partial T}{\partial \alpha }=\frac{2\sqrt{2\sqrt{3}s}}{\alpha \sqrt{\beta}}\left( \frac{\partial K}{\partial k}\frac{\partial k}{\partial \alpha }\sqrt{\alpha \sqrt{\beta}}- \frac{ \left(\sqrt{\beta}+\frac{\alpha }{2\sqrt{\beta}}\frac{\partial \beta}{\partial \alpha }\right)}{2\sqrt{\alpha \sqrt{\beta}}}K(k)\right).
	\end{equation}
On the other hand, deriving the functions $\beta=-4\alpha ^2-4\alpha +3+16r$ and $k$ in $(\ref{kfunctionalpha4})$ with respect to $\alpha$, we see
\begin{equation}
\label{derivqalpha4}
\frac{\partial q}{\partial \alpha }=-8\alpha -4,
\end{equation}
and 
$\frac{\partial k}{\partial \alpha }=\frac{6\alpha ^3+9\alpha ^2-18\alpha r+9r+48r^2}{\sqrt{3}k\alpha ^2\beta^\frac{3}{2}}.$
In addition, since $r=1-c^2>0$ and $\alpha >0$, we can rewrite \eqref{derivkalpha4} as
\begin{equation}
\label{derivkalpha4}\frac{\partial k}{\partial \alpha }=\frac{12\alpha ^3+14\alpha ^2+(2\alpha -9r)^2+18r+15r^2}{2\sqrt{3}k\alpha ^2\beta^\frac{3}{2}}.
\end{equation}
Notice that $\frac{\partial k}{\partial \alpha }$ is a positive function for all $(\alpha,c)\in\Omega$.

To simplify the calculations, let us denote $\sigma=6\alpha ^3+9\alpha ^2-18\alpha r+9r+48r^2$. Substituting \eqref{derivqalpha4} and \eqref{derivkalpha4} into \eqref{derivTalpha4}, it follows that $\dfrac{\partial T}{\partial \alpha }>0$ if, and only if 
$\frac{\partial K}{\partial k}\frac{\sigma}{\sqrt{3}k\alpha ^2\beta^\frac{3}{2}}\sqrt{\alpha \sqrt{\beta}}- \frac{ \left(\sqrt{\beta}+\frac{\alpha }{2\sqrt{\beta}}(-8\alpha -4)\right)}{2\sqrt{\alpha \sqrt{\beta}}}K(k)>0.$
Multiplying the last inequality by $\dfrac{\sqrt{3}\alpha ^\frac{3}{2}\beta^\frac{5}{4}}{\sigma}$, we have  $\dfrac{\partial T_\phi}{\partial \alpha }>0$ provided that
$\frac{\partial K}{\partial k}\frac{1}{k}-\left(\frac{\sqrt{3}\alpha \beta^\frac{1}{2}(\beta-4\alpha ^2-2\alpha )}{\sigma} \right)\frac{K(k)}{2}>0.$
Since the expression $\dfrac{\partial K}{\partial k}\dfrac{1}{k} -\dfrac{K}{2}$ is always positive for all $k\in (0,1)$, we can prove $\dfrac{\partial T}{\partial \alpha }>0$ by proving that $\frac{\sqrt{3}\alpha \beta^\frac{1}{2}(\beta-4\alpha ^2-2\alpha )}{\sigma}<1$. This fact can be easily checked using some massive calculations or it can be verified using Mathematica program. 


We prove now $T$ is increasing with respect to $c$. In fact, deriving $T$ in terms of $c$, we obtain
	
\begin{equation}\label{derivTc}\frac{\partial T}{\partial c}=\frac{2\sqrt{2}3^\frac{1}{4}\alpha ^\frac{1}{2}}{\alpha \sqrt{\beta}\sqrt{5-4c^2}\beta^\frac{3}{4}}\left[\left(-4cK(k)+(5-4c^2)\frac{\partial K}{\partial k}\frac{\partial k}{\partial c}\right)\beta -\frac{5-4c^2}{4}\frac{\partial \beta}{\partial c}K(k)   \right].\end{equation}
Since
$\frac{\partial \beta}{\partial c}=-32c$ and $\frac{\partial k}{\partial c}=\frac{2\sqrt{3}c}{k\alpha \beta^{\frac{3}{2}}}\left( 2\alpha +3+8(1-c^2)\right), $
it follows that
$$\frac{\partial T}{\partial c}=C\left[ \left(-4c\beta +\frac{(5-4c^2)(32c)}{4}\right)K(k)+\frac{2\sqrt{3}c(5-4c^2)\beta\left( 2\alpha +3+8(1-c^2)\right)}{k\alpha_2\beta^{\frac{3}{2}}}\frac{\partial K}{\partial k}\right],
$$	
where $C=\frac{2\sqrt{2}3^\frac{1}{4}\alpha ^\frac{1}{2}}{\alpha_2\sqrt{\beta}\sqrt{5-4c^2}\beta^\frac{3}{4}}>0.$
Let us denote  $C_1=2cC$. We obtain,
\begin{equation}\label{per1}\frac{\partial T}{\partial c}=C_1\left[ \left(-2\beta +4(5-4c^2)\right)K(k)+\frac{\sqrt{3}(5-4c^2)\left( 2\alpha +3+8(1-c^2)\right)}{k\alpha_2\beta^{\frac{1}{2}}}\frac{\partial K}{\partial k}\right].
\end{equation}	
Consider
 $\rho =\frac{\sqrt{3}(5-4c^2)\left( 2\alpha +3+8(1-c^2)\right)}{\alpha_2\beta^{\frac{1}{2}}}.$ We have by $(\ref{per1})$ that $\frac{\partial T}{\partial c}$ can be expressed as 
$\frac{\partial T}{\partial c}=C_1\rho \left[\frac{1}{k}\frac{\partial K}{\partial k} -\frac{2(2\beta -4(5-4c^2))}{\rho }\frac{K(k)}{2}\right].
$
Since $\rho >0$ and	$\frac{1}{k}\frac{\partial K}{\partial k} -\frac{K(k)}{2}>0$, it is enough to prove
$\omega:=\frac{2(2\beta -4(5-4c^2))}{\rho }<1.$
In fact, using that $\beta<3+12(1-c^2)$ and  $\alpha <\frac{-3+\sqrt{9+48(1-c^2)}}{4}$, it follows from the fact $\alpha >\frac{-1+\sqrt{1+4(1-c^2)}}{2}$ that
$\omega<\frac{(1+4(1-c^2))^\frac{1}{2}(-3+\sqrt{9+48(1-c^2)}) }{((1+4(1-c^2))^\frac{1}{2}+2+8(1-c^2))}.
$\\
\indent After some simplifications, we are enabled to conclude from the last inequality
\begin{equation}\label{eq01}
\omega<\frac{-3+\sqrt{12}\sqrt{1+4(1-c^2)} }{1+2\sqrt{1+4(1-c^2)}}.
\end{equation}
\indent Finally, denoting $e=\sqrt{1+4(1-c^2)}$, we have $0<e<\sqrt{5}$ with $g(e):=\frac{(-3+\sqrt{12}e)}{(1+2e)}$ being an increasing function with maximum value given by
\begin{equation}\label{eq02} g(\sqrt{5})=\frac{-3+2\sqrt{15}}{1+ 2\sqrt{5}}.
\end{equation}
Combining \eqref{eq01} and \eqref{eq02}, we obtain
$$\omega<g(\sqrt{5})<\frac{-3+2\sqrt{16}}{1+2\sqrt{4}}<1.$$
This conclude the proof of the lemma.
\end{proof}

\begin{theorem}\label{teo02}
	Let $L > 2\pi\sqrt{\frac{\sqrt{5}}{\sqrt{5}-1}}$ be fixed. For each 
	$c_0\in \left(0,\frac{1}{2}\sqrt{5-\frac{L^4}{(L^2-4\pi^2)^2}}\right)$
	consider the unique $\alpha_{4,0} \in \left(\frac{-1+\sqrt{1+4r}}{2}, \frac{-3+\sqrt{9+48r}}{4}\right)$ such that	$T(\alpha_{4,0}, c_0)=L$. Then,
	\begin{enumerate}
		\item there exist an interval $I_1$ around $c_0$, an interval $I_2$ around $\alpha_{4,0}$ and a unique smooth function $\Lambda:I_1 \to I_2$ such that $\Lambda(c_0)=\alpha_{4,0}$ and 
		$$T(\Lambda(c),c)=\frac{2\sqrt{2\sqrt{3}(5-4c^2)}K(k(\Lambda(c),c))}{\sqrt{\Lambda(c)\sqrt{\beta(\Lambda(c),c)}}}=L,$$
		where $c\in I_1$, $\alpha_4(c)=\Lambda(c)\in I_2$ and $k^2=k^2(c)\in (0,1)$ is given by \eqref{kfunctionalpha4},
		
		\item the dnoidal wave solution in \eqref{phix}, $\phi_c:=\phi(\cdot, \Lambda(c))$ determined by $\Lambda(c)$, has fundamental period $L$ and satisfies \eqref{ddedo1}. Moreover, the mapping
		$c\in I_1 \mapsto \phi_c \in H^n_{per}([0,L]),$
		is a smooth function for all $n\in \mathbb{N}$,
		
		\item The interval $I_1$ can be chosen as 
		$\left[0,\frac{1}{2}\sqrt{5-\frac{L^4}{(L^2-4\pi^2)^2}}\right).$
	\end{enumerate}
\end{theorem}
\begin{proof}
	The proof consists in applying the implicit function theorem. In fact, let us consider the open set $\Omega$ as in Lemma $\ref{teo01}$. By Lemma \ref{teo01}, we have that $\frac{\partial T}{\partial \alpha}>0$ for all $\alpha \in  \left(\frac{-1+\sqrt{1+4r}}{2}, \frac{-3+\sqrt{9+48r}}{4}\right)$. A simple application of the implicit function theorem enables us to deduce the existence of  an interval $I_1$ around $c_0$, an interval $I_2$ around $\alpha_{4,0}$ and a unique smooth function $\Lambda:I_1 \to I_2$ such that $\Lambda(c_0)=\alpha_{4,0}$ and $T(\Lambda(c),c)=L$ for all $c\in I_1$. We have then proved the first two items of the theorem.
		
		Since $c$ was chosen arbitrarily in the interval $\left(0,\frac{1}{2}\sqrt{5-\frac{L^4}{(L^2-4\pi^2)^2}}\right)$, the uniqueness of the function $\Lambda$ in terms of $c$ implies that the  $I_1$ can be extended to the whole interval $\left(0,\frac{1}{2}\sqrt{5-\frac{L^4}{(L^2-4\pi^2)^2}}\right)$.
\end{proof}

\begin{corollary}\label{corlambdac}
	Let $\Lambda:I_1 \to I_2$ be given by Theorem \ref{teo02}. Thus, $\Lambda$ is a strictly decreasing function in terms of $c$. 
\end{corollary}
\begin{proof}
	By Lemma \ref{teo01}, Theorem \ref{teo02} and the implicit function theorem, we can differentiate the equality $T(\Lambda(c),c)=L$ in terms of $c$ and to obtain
	$$\frac{\partial \Lambda}{\partial c}=-\frac{\frac{\partial T_\phi}{\partial c}}{\frac{\partial T_\phi}{\partial \alpha}}<0.$$
%

\end{proof}


\section{Characterization of all Positive and Periodic Waves.}

Let $c_0\in(0,1)$ be fixed. Consider $T$ the period-map defined in Lemma $\ref{teo01}$ but depending  only on $\alpha \in I_3=\left(\frac{-1+\sqrt{1+4r}}{2}, \frac{-3+\sqrt{9+48r}}{4}\right)$. According Lemma \ref{teo01}, we see that the period-map 
$$T: I_3 \rightarrow \left(\frac{2\pi \sqrt{s_0}}{(1+4r_0)^\frac{1}{4}\sqrt{\sqrt{1+4r_0}-1}}, +\infty \right),$$
is smooth and a strictly increasing function. As before, we have  $r_0=1-c_0^2$ and $s_0=5-4c^2_0$.

By $(\ref{Bfunctionalpha4})$, we see that the value of $B$ is given in terms of $\alpha\in I_3$ by
$$B=\frac{1}{5-4c_0^2}\left( \frac{\alpha^3}{3}+\frac{\alpha^2}{2}-(1-c_0^2)\alpha\right).$$
Next, we show that $\frac{\partial \alpha}{\partial B}>0$. In fact, the derivative of $B$ with respect to $\alpha$ is
$\frac{\partial B}{\partial \alpha}=\frac{1}{5-4c_0^2}\left(\alpha^2+\alpha-(1-c_0^2)\right).$
Then, the zeros of $\frac{\partial B}{\partial \alpha}$ are 
$\frac{-1-\sqrt{1+4(1-c_0^2)}}{2}$ and $\frac{-1+\sqrt{1+4(1-c_0^2)}}{2}.$
Since $s_0=5-4c_0^2>0$ and $\alpha \in \left(\frac{-1+\sqrt{1+4r_0}}{2}, \frac{-3+\sqrt{9+48r_0}}{4}\right)$, it follows that $B$ is strictly increasing in terms of $\alpha$ whose minimum and maximum values are given respectively by
$B\left(\frac{-1+\sqrt{1+4r_0}}{2}\right)=\frac{1-(5-4c_0^2)^{\frac{3}{2}}+6(1-c_0^2)}{12(5-4c_0^2)}=:B_{c_0}$
and 
$B\left(\frac{-3+\sqrt{9+48r}}{4} \right)=0.$ Thus, by the inverse function theorem we obtain 
$$\alpha: (B_{c_0},0) \rightarrow \left(\frac{-1+\sqrt{1+4r_0}}{2}, \frac{-3+\sqrt{9+48r_0}}{4}\right),$$
 is a smooth and $\frac{\partial \alpha}{\partial B}>0$.

Therefore, one has
$$\frac{\partial T}{\partial B}=\frac{\partial T}{\partial \alpha}\frac{\partial \alpha}{\partial B}>0,$$
that is, we have proved the following result.
\begin{proposition}\label{prop1}
	Let $c_0\in(0,1)$ be fixed. The period map
	$$T: (B_{c_0},0) \rightarrow \left(\frac{2\pi \sqrt{s_0}}{(1+4r_0)^\frac{1}{4}\sqrt{\sqrt{1+4r_0}-1}}, +\infty \right)$$
	is a smooth function in terms of $B\in(B_{c_0},0)$ and $\frac{\partial T}{\partial B}>0$.
\end{proposition}
\begin{flushright}
$\square$
\end{flushright}

\indent Next, as far as we know, we can use the standard ODE theory to conclude that for $\alpha\in I_3$, all positive even periodic waves satisfy the initial value problem (IVP), 
\begin{equation}\label{eq06}\left\{ \begin{array}{l}
-\phi''(x)+\frac{r_0}{s_0}\phi(x) -\frac{1}{s_0}(\phi^3(x)+\phi^5(x))=0,\\
\phi(0)=\sqrt{\alpha},\\
\phi'(0)=0.
\end{array}\right.\end{equation}

Proposition \ref{prop1} is then used to conclude the following result.

\begin{proposition}
Let $c_0\in(0,1)$ be fixed. All positive and even periodic solution associated to the equation \eqref{ddedo1} has the dnoidal profile given by \eqref{phix}.
\end{proposition}
\begin{proof}
	Let $L> 2\pi\sqrt{\frac{\sqrt{5}}{\sqrt{5}-1}}$ be fixed. By the monotonicity of the period map $T: (B_{c_0},0) \rightarrow \left(\frac{2\pi \sqrt{s_0}}{(1+4r_0)^\frac{1}{4}\sqrt{\sqrt{1+4r_0}-1}}, +\infty \right)$ determined in Proposition \ref{prop1}, there exists a unique $B_0 \in (B_{c_0},0)$ such that $T(B_0)=L$. Moreover, the monotonicity of $\alpha(B)$ in terms of $B$ so implies the existence of a unique $\alpha_0\in I_3$ such that $B_0=\alpha(\alpha_0)$. The proof then follows by Theorem \ref{teo02}.
\end{proof}

\section{Spectral Properties.}\label{sec4}

As mentioned in the introduction, the stability comes from the positivity of $\mathcal{D}$, except in one null and in one negative direction.  In this context the spectral analysis of $\mathcal{D}$ turns crucial for obtaining stability. On the other hand, considering 
$$A=\left(\begin{array}{cc}
	I & 0\\
	-cI& I
\end{array}\right),$$
it follows that the spectral analysis of $\mathcal{D}$ can be reduced to the analysis of $\mathcal{L}:=A^t\mathcal{D}A$, that is
\begin{equation}\mathcal{L}=\left(\begin{array}{cc}
		\mathcal{L}_1 & 0\\
		0& \mathcal{L}_2
	\end{array}\right) = \left( \begin{array}{cc}
		-s\partial_x^2+r-3\phi^2-5\phi^4& 0 \\
		0 & I-4\partial_x^2  \\
	\end{array}\right).\label{L-35i}
\end{equation}

In order of obtaining a characterization of the nonpositive spectrum of $\mathcal{L}$, given by $(\ref{L-35i})$, we recall some basic facts about Floquet's theory (for further details, see \cite{est} and \cite{magnus}). Let $Q$ be a smooth even $T$-periodic function.  Denote by $P$ the Hill operator defined in $L_{per}^2([0,T])$, with domain $D(P)=H_{per}^2([0,T])$ as
$$
P=-\partial_x^2+Q(x).
$$
As far as we know, the spectrum of $P$ is formed by an unbounded sequence of
real eigenvalues arranged as follows
\begin{equation}\label{seq}
	\lambda_0 < \lambda_1 \leq \lambda_2 < \lambda_3 \leq \lambda_4<
	\cdots\; < \lambda_{2n-1} \leq \lambda_{2n}\; \cdots,
\end{equation}
where the equality means that $\lambda_{2n-1} = \lambda_{2n}$  is a
double eigenvalue. According with the classical Oscillation Theorem, we see that the spectrum of $P$ can be characterized by the number of zeros
of the corresponding eigenfunctions. In fact, if $\varphi$ is an eigenfunction associated to either $\lambda_{2n-1}$ or $\lambda_{2n}$, then $\varphi$ has exactly
$2n$ zeros in the half-open
interval $[0, T)$. In particular, the even eigenfunction associated to the first eigenvalue $\lambda_0$ has no zeros in $[0, T]$.

Let $\varphi$ be a non-trivial $T$-periodic solution of the Hill equation
\begin{equation}\label{zeqL}
	-f''+Q(x)f=0.
\end{equation}
If $y$ is a solution of \eqref{zeqL} which is linearly independent with $\varphi$, there exists a constant $\theta$ (depending on $y$ and $\varphi$) such that
\begin{equation}\label{theta1}
	y(x+T)=y(x)+\theta \varphi(x).
\end{equation}
Thus, $\theta=0$ implies that $y$ is a periodic solution for $(\ref{zeqL}).$

Next result gives that it is possible to decide the exact position of the zero eigenvalue by knowing the precise sign of $\theta$ in $(\ref{theta1})$.

\begin{lemma}\label{specprop}
	Let $\theta$ be the constant given by (\ref{theta1}) and suppose that $\varphi$ is an $T-$periodic solution for the equation (\ref{zeqL}) containing only two zeros over $[0,T).$ The
	eigenvalue $\lambda=0$ is simple if and only if $ \theta \neq 0$.
	Moreover, if $\theta \neq 0$, then  $ \lambda_{1}=0$ if $\theta <
	0$, and $ \lambda_{2}=0$ if $\theta > 0$.
\end{lemma}
\begin{proof}
	See \cite{natali1} and \cite{neves}.
\end{proof}

\indent Before giving the behaviour of the non-positive spectrum of $\mathcal{L}_1$, we need some preliminary tools. First of all for the solution $\phi$ obtained by Theorem $\ref{teo02}$, we have that $\phi'$ solves the Hill equation $\mathcal{L}_1\phi'=0$ just by deriving $(\ref{ddedo})$ with respect to $x$. In addition, the dnoidal solution in $(\ref{phix})$ is positive and $\phi'$ is an odd function having two zeros in the interval $[0,L)$. By the classical Floquet theory, we see that $\lambda_1=0$ or $\lambda_2=0$. We show that $\lambda_1=0$ and it results to be simple. To this end, we first present the basica lemma
\begin{lemma} \label{teo13}
	We have, $\frac{dT}{dB}=-\frac{\theta}{2}$, where $\theta$ is the constant in (\ref{theta1}).
\end{lemma}
\begin{proof}
	In a general setting, consider $\{\phi',y\}$ a fundamental set of solutions for the Hill equation $\mathcal{L}y=0$, where $y\in C^{\infty}([0,T])$. Thus, we have that $\bar{y}$ and $\phi'$ satisfies the equation $(\ref{zeqL})$. In addition, the Wronskian $W$ associated to fundamental set satisfies $W(\phi'(x),\bar{y}(x))=1$ for all $x\in[0,T]$. Since $\phi'$ is odd and periodic, we obtain that $y$ is even and it satisfies the following IVP
	\begin{equation}\left\{
		\begin{array}{l}
			\displaystyle -(5-4c^2)y''+ (1-c^2) y-3\phi^2y-5\phi^4y = 0 ,\\
		 y(0) =-\frac{1}{\phi''(0)}, \\
			y'(0)= 0.
		\end{array} \right.
		\label{y}
	\end{equation}	
	The smoothness of $\phi'$ in terms of the parameter $B$ enables us to take the derivative of $\phi'(T)=0$ with respect to $B$ to obtain
	\begin{equation}\label{eq12343}
		\phi''(T)\frac{dT}{dB}+\frac{\partial \phi'}{\partial B}(T)=0.\end{equation}
	\indent
	Deriving equation $(\ref{eqquadrature})$ with respect to $B$ and taking $x=0$ in the final result, we obtain by $(\ref{ddedo})$ at the point $x=0$ that
	$\frac{\partial \phi}{\partial B}(0)=-\frac{1}{ 2\phi''(0)}$. In addition, since $\phi'$ is odd one has that $\frac{\partial \phi'}{\partial B}$ is also odd and thus, $\frac{\partial \phi'}{\partial B}(0)=0$. The existence and uniqueness theorem for classical ODE applied to the problem $(\ref{y})$ enables us to deduce that $y=2\frac{\partial \phi}{\partial B}$. Therefore, we can combine $(\ref{theta1})$ with $(\ref{eq12343})$ to obtain that $\frac{dT}{dB}=-\frac{\theta}{2}$.
\end{proof}

\indent By Theorem $\ref{teo02}$, Lemma $\ref{specprop}$, and Lemma $\ref{teo13}$, we obtain the following result about the non-positive spectrum of $\mathcal{L}$.
\begin{lemma}\label{lema167}
	Let $L > 2\pi\sqrt{\frac{\sqrt{5}}{\sqrt{5}-1}}$ be fixed and consider  
	$c\in \left(0,\frac{1}{2}\sqrt{5-\frac{L^4}{(L^2-4\pi^2)^2}}\right)$. Operator
	$\mathcal{L}_1$
	in defined in $L_{per}^2$ with domain in
	$H_{per}^2$ has a unique negative eigenvalue
	which is simple and zero is a simple eigenvalue with eigenfunction
	$\phi'$.  Moreover, the remainder of the spectrum of $\mathcal{L}_1$
	is constituted by a discrete set
	of eigenvalues bounded away from zero.
\end{lemma}\vspace{-16pt}
\begin{flushright}
	$\square$
\end{flushright}

\indent Since $\mathcal{L}_2$ in $(\ref{L-35i})$ is positive, we conclude by Lemma $\ref{lema167}$ and the fact that $\mathcal{L}=A^{t}\mathcal{D}A$ the following result.

\begin{lemma}\label{lema1678}
	Let $L > 2\pi\sqrt{\frac{\sqrt{5}}{\sqrt{5}-1}}$ be fixed and consider  
	$c\in \left(0,\frac{1}{2}\sqrt{5-\frac{L^4}{(L^2-4\pi^2)^2}}\right)$. Operator
	$\mathcal{D}$
	in defined in $\mathbb{L}_{per}^2$ with domain in
	$\mathbb{H}_{per}^2$ has a unique negative eigenvalue
	which is simple and zero is a simple eigenvalue with eigenfunction
	$(\phi',-c\phi')$.  Moreover, the remainder of the spectrum of $\mathcal{D}$
	is constituted by a discrete set
	of eigenvalues bounded away from zero.
\end{lemma}\vspace{-16pt}
\begin{flushright}
	$\square$
\end{flushright}

\section{Spectral Stability Results - Proof of Theorem $\ref{mainT}$.}

In this section we present our spectral stability result and consequently, the proof of Theorem $\ref{mainT}$. To ensure the existence of the pair $(p,q)$ in $(\ref{changevar1})$, it is necessary to establish the existence of a local solution $(u,v)$ for the Cauchy problem associated with the system $(\ref{sysdde})$ 

\begin{equation}\label{cauchysysdde}
	\left\{\begin{array}{llll}u_t=v_x,\\
		v_t=(I-4\partial_x^2)^{-1}\partial_x((I-5\partial_x^2)u-(u^3+u^5))\ \mbox{in}\ [0,L]\times (0,+\infty),\\
		(u(x,0),v(x,0))=(u_0(x),v_0(x))\ \mbox{in}\ [0,L].
	\end{array}\right.\end{equation}

 In fact, we have the following result 

\begin{lemma} There exists a time $t_0>0$ such that for all initial $(u_0,v_0)\in \mathbb{H}_{per}^1$ the Cauchy problem in $(\ref{cauchysysdde})$ has a unique solution $(u,v)$ defined in $C([0, t_0),\mathbb{H}_{per}^1)$. In addition, the pair $(u,v)$ satisfies the conserved quantities in $(\ref{E})-(\ref{H})$.	
	\end{lemma}
\begin{proof}
	The proof of this result has the same spirit as in \cite[Section 2]{wang} and because of this, we omit the details.
\end{proof}

Next, we need to calculate the difference 
$n(\mathcal{D}_{\Pi})-n(\mathcal{P})$, where
$\mathcal{D}_{\Pi}$ is defined as in $(\ref{operatorproj})$. 
 Again and to improve the comprehension of the reader, matrix $\mathcal{P}$ is given by 
 $$
\mathcal{P}=\left(\begin{array}{llll}\langle \mathcal{D}^{-1}(-c\chi,\chi),(-c\chi,\chi)\rangle_{\mathbb{L}_{per}^2} & \langle \mathcal{D}^{-1}(1,0),(-c\chi,\chi)\rangle_{\mathbb{L}_{per}^2} & \langle \mathcal{D}^{-1}(0,1),(-c\chi,\chi)\rangle_{\mathbb{L}_{per}^2}\\\\
	\langle \mathcal{D}^{-1}(-c\chi,\chi),(1,0)\rangle_{\mathbb{L}_{per}^2} & \langle \mathcal{D}^{-1}(1,0),(1,0)\rangle_{\mathbb{L}_{per}^2} & \langle \mathcal{D}^{-1}(0,1),(1,0)\rangle_{\mathbb{L}_{per}^2}\\\\
	\langle \mathcal{D}^{-1}(-c\chi,\chi),(0,1)\rangle_{\mathbb{L}_{per}^2} & \langle \mathcal{D}^{-1}(1,0),(0,1)\rangle_{\mathbb{L}_{per}^2} & \langle \mathcal{D}^{-1}(0,1),(0,1)\rangle_{\mathbb{L}_{per}^2}\end{array}\right),
$$
where $\chi=4\phi''-\phi$.\\
\indent In order to calculate $n(\mathcal{D}_{\Pi})$, we need to use the Index Theorem for self-adjoint operators  (see \cite[Theorem 5.3.2]{kapitula} that gives a precise counting of the spectral information concerning $\mathcal{D}_{\Pi}$ in terms of the spectral properties associated to $\mathcal{D}$. More precisely, since $\ker(\mathcal{D})=[(\phi',-c\phi')]$, we have 
\begin{equation}\label{indexformula12}
	\text{n}(\mathcal{D}_{{\Pi}})=\text{n}(\mathcal{D})-n_0-z_0
\end{equation}
and
\begin{equation}\label{indexformula123}
	\text{z}(\mathcal{D}_{{\Pi}})=\text{z}(\mathcal{D})+z_0,
\end{equation} 
where $\text{z}(\mathcal{A})$ indicates the dimension of the kernel of a certain linear operator $\mathcal{A}$. Parameters $n_0$ and $z_0$ are, respectively, the number of negative eigenvalues and the dimension of the kernel associated to the matrix
$$
\mathcal{Q}=\left(\begin{array}{llll}\langle \mathcal{D}^{-1}(1,0),(1,0)\rangle_{\mathbb{L}_{per}^2} & \langle \mathcal{D}^{-1}(0,1),(1,0)\rangle_{\mathbb{L}_{per}^2}\\\\
	 \langle \mathcal{D}^{-1}(1,0),(0,1)\rangle_{\mathbb{L}_{per}^2} & \langle \mathcal{D}^{-1}(0,1),(0,1)\rangle_{\mathbb{L}_{per}^2}\end{array}\right).
$$

\indent Next, we need to calculate all the entries of the matrices $\mathcal{P}$ and $\mathcal{Q}$ above. In fact, to calculate $\mathcal{Q}$ the procedure is the following: since $\{(1,0),(0,1)\}\subset \ker(\mathcal{D})^{\bot}$ and $\mathcal{D}:\ker(\mathcal{D})^{\bot}\rightarrow \ker(\mathcal{D})^{\bot}$ is invertible, we obtain the existence of unique $(f_1,g_1),\ (f_2,g_2)\in \mathbb{H}_{per}^2$ such that $\mathcal{D}(f_1,g_1)=(1,0)$ and $\mathcal{D}(f_2,g_2)=(0,1)$. Both cases mean, after some calculations
\begin{equation}\label{range1}
	-(5-4c^2)\partial_x^2f_1+(1-c^2)f_1-(3\phi^2+5\phi^4)f_1=1,
	\end{equation}
and
\begin{equation}\label{range2}
	-(5-4c^2)\partial_x^2f_2+(1-c^2)f_2-(3\phi^2+5\phi^4)f_2=-c.
\end{equation}
\indent Since $c\neq0$, we deduce from the uniqueness of $f_1$ that $f_2=-cf_1$. Therefore, matrix $\mathcal{Q}$ becomes 
$$
\mathcal{Q}=\left(\begin{array}{cccc}\langle f_1,1\rangle_{{L}_{per}^2} & -c\langle f_1,1\rangle_{{L}_{per}^2}\\\\
	-c\langle f_1,1\rangle_{{L}_{per}^2} & L+c^2\langle f_1,1\rangle_{{L}_{per}^2}\end{array}\right).
$$
\indent On the other hand, since matrix $\mathcal{Q}$ is a $2\times 2$ block of the matrix $\mathcal{P}$, we only need to calculate the inner products $\langle \mathcal{D}^{-1}(-c\chi,\chi),(-c\chi,\chi)\rangle_{\mathbb{L}_{per}^2}$, $\langle \mathcal{D}^{-1}(-c\chi,\chi),(1,0)\rangle_{\mathbb{L}_{per}^2}=	\langle \mathcal{D}^{-1}(1,0),(-c\chi,\chi)\rangle_{\mathbb{L}_{per}^2}$, and $\langle \mathcal{D}^{-1}(-c\chi,\chi),(0,1)\rangle_{\mathbb{L}_{per}^2}=	\langle \mathcal{D}^{-1}(0,1),(-c\chi,\chi)\rangle_{\mathbb{L}_{per}^2}$. In fact, since $\mathcal{D}^{-1}(-c\chi,\chi)=(-c4\phi''+c\phi,4\phi''-\phi)=\left(\frac{\partial\phi}{\partial c},-\phi-c\frac{\partial\phi}{\partial c}\right)$ we obtain, after some computations,

	\begin{equation}\label{dddc}
	\langle \mathcal{D}^{-1}(-c\chi,\chi),(-c\chi,\chi)\rangle_{\mathbb{L}_{per}^2}=\int_0^L \left[\phi^2+4(\phi')^2+2c\phi \frac{\partial \phi}{\partial c}+8c\phi'\frac{\partial \phi'}{\partial c}\right]dx,	
	\end{equation}

\begin{equation}\label{dddc1}
	\langle \mathcal{D}^{-1}(-c\chi,\chi),(1,0)\rangle_{\mathbb{L}_{per}^2}=\int_0^L\frac{\partial\phi}{\partial c}dx,
		\end{equation}
	
	and 
\begin{equation}\label{dddc2}
	\langle \mathcal{D}^{-1}(-c\chi,\chi),(0,1)\rangle_{\mathbb{L}_{per}^2}=-\int_0^L\phi dx-c\int_0^L\frac{\partial\phi}{\partial c}dx
\end{equation}	
	
\indent From equations \eqref{dddc}, \eqref{dddc1}, and \eqref{dddc2}, we can conclude that all three inner products mentioned above can be calculated by knowing the behavior of $\frac{\partial \phi}{\partial c}$. To this end, we need to derive equation \eqref{ddedo1} with respect to $c$ in order to obtain
$$-(5-4c^2)\left(\frac{\partial \phi}{\partial c}\right)''+(1-c^2)\frac{\partial \phi}{\partial c}-3\phi^2\frac{\partial \phi}{\partial c}-5\phi^4\frac{\partial \phi}{\partial c}=-8c\phi''+2c\phi,$$
that is,
\begin{equation}
\label{eq04}
\mathcal{L}_1\left(\frac{\partial \phi}{\partial c}\right)=-8c\phi''+2c\phi.
\end{equation}
Since the derivative of $\phi$ in terms of $c$ preserves the parity in terms of the variable $x\in [0,L]$, we obtain that  $\frac{\partial \phi}{\partial c}$ is the solution of the following IVP
\begin{equation}\left\{
	\begin{array}{lllll}
		\displaystyle \mathcal{L}_1\left(\frac{\partial \phi}{\partial c}\right)=-8c\phi''+2c\phi ,\\
	\displaystyle	\frac{\partial \phi}{\partial c}(0)=a_c, \\\\
\displaystyle	\frac{\partial \phi'}{\partial c}(0)=0.
	\end{array} \right.
	\label{edodphi}
\end{equation}	
\indent Next, we present the exact value of $a_c$ to determine numerically the function $\frac{\partial \phi}{\partial c}$ in terms of $x\in [0,L]$. As a consequence of this fact, we obtain convenient expressions for the inner products in \eqref{dddc}, \eqref{dddc1} and \eqref{dddc2}.

Let $y$ be the non-periodic solution of the homogeneous problem associated to the IVP given by \eqref{y}. Multiplying the first equation in $(\ref{edodphi})$ by $y$, we obtain after two integration by parts and using $\frac{\partial \phi'}{\partial c}(L)=\frac{\partial \phi'}{\partial c}(0)=0$, $\frac{\partial \phi}{\partial c}(L)=\frac{\partial \phi}{\partial c}(0)$ and $y'(0)=0$ that

\begin{equation}\label{eq03}
\int_0^L\mathcal{L}\left(\frac{\partial \phi}{\partial c}\right) y dx = (5-4c^2)\frac{\partial \phi}{\partial c}(0)y'(L).\end{equation}

Since $y$ is non-periodic with $y'(0)=0$ and $y(0)=y(L)$, we see that $y'(L)\neq0$ and we obtain by the first equation in $(\ref{edodphi})$, the following formula for $a_c$ given by
\begin{equation}
\label{eq05}
a_c = -\frac{2c}{(5-4c^2)y'(L)}\int_0^L\chi y dx.
\end{equation}

\indent On the other hand, to calculate the inner product $\langle f_1,1\rangle _{L_{per}^2}$ present in the matrix $\mathcal{Q}$, we need to use a similar procedure as determined to obtain the value $a_c$ in $(\ref{eq05})$. Indeed, by $(\ref{range1})$, we see that $\mathcal{L}f_1=1$. Multiplying the equality $\mathcal{L}f_1=1$ by the non-periodic solution $y$ of the IVP $(\ref{y})$, and integrating the result over the interval $[0,L]$, we obtain after two integration by parts that 
\begin{equation}\label{f10}
	f_1(0) = \frac{1}{(5-4c^2)y'(L)}\int_0^L y dx.
	\end{equation}
Therefore, $f_1$ solves the IVP 

 \begin{equation}\left\{
 	\begin{array}{lllll}
 		\displaystyle \mathcal{L}_1f_1=1 ,\\
 	 f_1(0)=\frac{1}{(5-4c^2)y'(L)}\int_0^L y dx, \\
 			f_1'(0)=0.
 	\end{array} \right.
 	\label{edof1}
 \end{equation}	
\indent Finally, we need to solve the initial value problems $(\ref{edodphi})$ and 
$(\ref{edof1})$ numerically. Important to mention that both numerical results are useful to calculate the inner products present in the entries of the matrices $\mathcal{P}$ and $\mathcal{Q}$. Our intention is to obtain, for fixed values of $L>2\pi\sqrt{\frac{\sqrt{5}}{\sqrt{5}-1}}$, the behaviour of the determinants of $\mathcal{P}$ and $\mathcal{Q}$ in terms of the wave speed $c\in \left(0,\frac{1}{2}\sqrt{5-\frac{L^4}{(L^2-4\pi^2)^2}}\right)$. We first consider the determinant of the matrix $\mathcal{Q}$. Using Mathematica program, we can solve it to conclude that the $\det(\mathcal{Q})$ is always positive according to the Figure $\ref{fig1}$. This implies that, from Lemma $\ref{lema1678}$ and equalities $(\ref{indexformula12})$ and $(\ref{indexformula123})$, we have
${\rm n}(\mathcal{D}_{\Pi})=n(\mathcal{D})-n_0-z_0=1$ and ${\rm z}(\mathcal{D}_{\Pi})=z(\mathcal{D})+z_0=1$.

\begin{figure}[htb]
	\centering
	\includegraphics[scale=1]{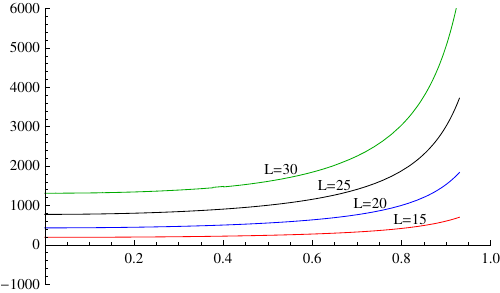}
	\caption{Plots for the function $\det(\mathcal{Q})$ in terms of $c\in\left(0,\frac{1}{2}\sqrt{5-\frac{L^4}{(L^2-4\pi^2)^2}}\right)$ for fixed values of $L>2\pi\sqrt{\frac{\sqrt{5}}{\sqrt{5}-1}}$.}
\label{fig1}\end{figure}

\indent Next picture shows the behaviour of $\det(\mathcal{P})$. According to the Figure $\ref{fig2}$, we see that for a fixed $L>0$, 
there exists an $c(L)>0$ such that if $c\in(0,c(L)]$, the periodic wave  
$\overrightarrow{\phi_c}=(\phi,-c\phi)$ is spectrally unstable, and if $c\in \left(c(L),\frac{1}{2}\sqrt{5-\frac{L^4}{(L^2-4\pi^2)^2}}\right)$, the periodic wave  
$\overrightarrow{\phi_c}=(\phi,-c\phi)$ is spectrally stable. Numerically, we can determine some values for the threshold value $c(L)>0$. In fact, if $L=15$, we obtain $c(L)\approx 0.5511026$, and if $L=20$, we obtain $c(L)\approx 0.5910113$. Finally, if $L=25$, we have $c(L)\approx 0.6100580$, and for $L=30$, we have $c(L)\approx 0.6217128$. Theorem $\ref{mainT}$ is then proved.

\begin{figure}[htb]
\centering
\includegraphics[scale=1]{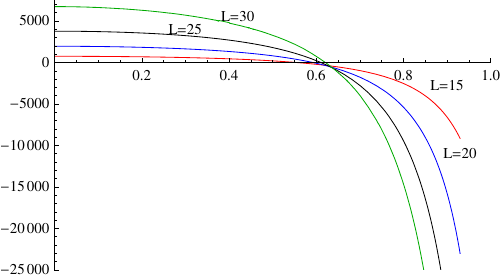}
\caption{Plots for the function $\det(\mathcal{P})$ in terms of $c\in\left(0,\frac{1}{2}\sqrt{5-\frac{L^4}{(L^2-4\pi^2)^2}}\right)$ for fixed values of $L>2\pi\sqrt{\frac{\sqrt{5}}{\sqrt{5}-1}}$. We see in all pictures that function $\det(\mathcal{P})$ has a zero depending on the value of $L$.}
\label{fig2}\end{figure}

\end{document}